\documentclass[12pt,reqno]{amsart}
\usepackage{latexsym,amsmath,amsfonts,amssymb}
\theoremstyle{plain}
\newtheorem{Thm}{Theorem}
\newtheorem{Lem}{Lemma}

\theoremstyle{definition}

\theoremstyle{remark}

\def\1{{\bf 1}}

\def\jacob #1#2{\genfrac{(}{)}{}{}{#1}{#2}}
\def\pmod #1{\ ({\rm{mod}}\ #1)}

\begin{document}

\medskip
\title{On the divisibility of the  truncated hypergeometric function ${}_{3}F_2$}
\author{Hao Pan}
\address{Department of Mathematics, Nanjing University, Nanjing 210093,
People's Republic of China}
\email{haopan79@yahoo.com.cn}
\author{Yong Zhang}
\address{Department of Basic Course, Nanjing Institute of Technology, Nanjing 211167,
People's Republic of China}
\email{yongzhang1982@163.com}
\begin{abstract}
We prove a result on the divisibility of the  truncated hypergeometric function ${}_{3}F_2\Big[\genfrac{}{}{0pt}{}{\alpha}{}\genfrac{}{}{0pt}{}{\beta}{1}\genfrac{}{}{0pt}{}{1-\alpha-\beta}{1}\Big|1\Big]$.
\end{abstract}

\maketitle

\vskip 10pt

The hypergeometric function ${}_{q+1}F_q\Big[\genfrac{}{}{0pt}{}{\alpha_1}{}\genfrac{}{}{0pt}{}{\alpha_2}{\beta_1}\genfrac{}{}{0pt}{}{\ldots}{\ldots}\genfrac{}{}{0pt}{}{\alpha_{q+1}}{\beta_q}\Big|x\Big]$  is given by
$$
{}_pF_q\bigg[\begin{matrix}\alpha_1&\alpha_2&\ldots&\alpha_p\\\beta_1&\beta_2&\ldots&\beta_q\end{matrix}\bigg|x\bigg]=\sum_{k=0}^\infty\frac{(\alpha_1)_k(\alpha_2)_k\cdots (\alpha_p)_k}{(\beta_1)_k(\beta_2)_k\cdots (\beta_q)_k}\cdot\frac{x^k}{k!},
$$
where 
$$
(\alpha)_k=\begin{cases} \alpha(\alpha+1)\cdots(\alpha+k-1),\quad&\text{if }k\geq 1,\\
1,&\text{if }k=0.
\end{cases}
$$

Define the truncated hypergeometric function 
$${}_pF_q\bigg[\begin{matrix}\alpha_1&\alpha_2&\ldots&\alpha_p\\\beta_1&\beta_2&\ldots&\beta_q\end{matrix}\bigg|x\bigg]_n=
\sum_{k=0}^{n-1}\frac{(\alpha_1)_k(\alpha_2)_k\cdots (\alpha_p)_k}{(\beta_1)_k(\beta_2)_k\cdots (\beta_q)_k}\cdot\frac{x^k}{k!}.
$$
In \cite{M}, Mortenson studied the arithmetical  properties of ${}_{2}F_1\Big[\genfrac{}{}{0pt}{}{m/d}{}\genfrac{}{}{0pt}{}{(d-m)/d}{1}\Big|x\Big]$,
where $1\leq m<d$ and $p$ is a prime with $p\equiv1\pmod{d}$. With help of the Gross-Koblitz formula, he proved that for prime $p\geq 5$,
$$
{}_2F_1\bigg[\begin{matrix}1/2&1/2\\&1\end{matrix}\bigg|1\bigg]_p\equiv\jacob{-1}{p}\pmod{p^2},\qquad{}_2F_1\bigg[\begin{matrix}1/3&2/3\\&1\end{matrix}\bigg|1\bigg]_p\equiv\jacob{-3}{p}\pmod{p^2},
$$$$
{}_2F_1\bigg[\begin{matrix}1/4&3/4\\&1\end{matrix}\bigg|1\bigg]_p\equiv\jacob{-2}{p}\pmod{p^2},\qquad{}_2F_1\bigg[\begin{matrix}1/6&5/6\\&1\end{matrix}\bigg|1\bigg]_p\equiv\jacob{-1}{p}\pmod{p^2},
$$
where $\jacob{\cdot}{p}$ denotes the Legendre symbol modulo $p$. Subsequently, Sun \cite{S} gave the elementary proofs for the above congruences. However, in fact, in \cite{S2}, Sun proved that
$$
{}_2F_1\bigg[\begin{matrix}\alpha&1-\alpha\\&1\end{matrix}\bigg|1\bigg]_p\equiv (-1)^{\{\alpha\}_p-1}\pmod{p^2},
$$
where $\{\alpha\}_p$ denotes the least non-negative residue of $\alpha$ modulo $p$.

Motivated by the above results, in this note, we shall study the truncated hypergeometric function ${}_{3}F_2\Big[\genfrac{}{}{0pt}{}{\alpha}{}\genfrac{}{}{0pt}{}{\beta}{1}\genfrac{}{}{0pt}{}{1-\alpha-\beta}{1}\Big|1\Big]$.
\begin{Thm}\label{t2} Suppose that $p$ is an odd prime and $\alpha,\beta$ are $p$-integral. 
If $\{\alpha\}_p,\{\beta\}_p\geq 1$ and $\{\alpha\}_p+\{\beta\}_p\leq p$,
then
$$
{}_3F_2\bigg[\begin{matrix}\alpha&\beta&1-\alpha-\beta\\&1&1\end{matrix}\bigg|1\bigg]_p\equiv0\pmod{p^2}.
$$
\end{Thm}
\begin{proof}
By replacing $\alpha$ and $\beta$ by $-\alpha$ and $-\beta$, (\ref{t1e}) is equivalent to
$$
\sum_{k=0}^{p-1}(-1)^k\binom{\alpha}{k}\binom{\beta}{k}\binom{-1-\alpha-\beta}{k}\equiv0\pmod{p^2},
$$
where $\{\alpha\}_p+\{\beta\}_p\geq p$.
Let us recall the Saalsch\"utz identity:
$$
{}_3F_2\bigg[\begin{matrix}a&b&c\\ &d&e\end{matrix}\bigg|1\bigg]=\frac{(d-a)_{|c|}(d-b)_{|c|}}{(d)_{|c|}(d-a-b)_{|c|}},
$$
where $c$ is a non-positive integer and
$
d+e=a+b+c+1$.

Below assume that $\alpha+\beta\geq p$. If $1\leq \alpha,\beta\leq p-1$, then by the Saalsch\"utz identity, we have
\begin{align*}
\sum_{k=0}^{p-1}(-1)^k\binom{\alpha}{k}\binom{\beta}{k}\binom{-1-\alpha-\beta}{k}\equiv&.\sum_{k=0}^{\infty}\frac{(-\alpha)_k(-\beta)_k(1+\alpha+\beta)_k}{(k!)^3}\\
=&\frac{(1+\beta)_{\alpha}(-\alpha-\beta)_{\alpha}}{(1)_{\alpha}(-\alpha)_{\alpha}}=\binom{\alpha+\beta}{\alpha}^2\equiv0\pmod{p^2}.
\end{align*}
Thus it suffices to prove that for any integer $s$
$$
\sum_{k=0}^{p-1}(-1)^k\binom{\alpha}{k}\binom{\beta}{k}\binom{-1-\alpha-\beta}{k}\equiv \sum_{k=0}^{p-1}(-1)^k\binom{\alpha}{k}\binom{\beta+sp}{k}\binom{-1-\alpha-\beta-sp}{k}\pmod{p^2},
$$
provided that $0\leq\alpha,\beta\leq p-1$ and $\alpha+\beta\geq p$.

Evidently,
\begin{align}
\binom{\alpha+sp}{k}&=\sum_{j=0}^{k}\binom{\alpha}{k-j}\binom{sp}{j}=\binom{\alpha}{k}+\sum_{j=1}^{k}\binom{\alpha}{k-j}\binom{(s-1)p+p-1}{j-1}\frac{sp}{j}\notag\\
&\equiv{\binom{\alpha}{k}+\sum_{j=1}^{k}\binom{\alpha}{k-j}\binom{s-1}{0}\binom{p-1}{j-1}\frac{sp}{j}}\notag\\
&\equiv{\binom{\alpha}{k}+sp\sum_{j=1}^{k}\binom{\alpha}{k-j}\frac{(-1)^{j-1}}{j}}\pmod{p^2}.
\end{align}
Then we have
\begin{align*}
&\sum_{k=0}^{\alpha}(-1)^k\binom{\alpha}{k}\binom{\beta+sp}{k}\binom{-1-\alpha-\beta-sp}{k}-\sum_{k=0}^{\alpha}(-1)^k\binom{\alpha}{k}\binom{\beta}{k}\binom{-1-\alpha-\beta}{k}\\
\equiv&
sp\sum_{k=1}^{\alpha}(-1)^k\binom{\alpha}{k}\sum_{j=1}^k\frac{(-1)^j}{j}\bigg(\binom{\beta}{k}\binom{-1-\alpha-\beta}{k-j}-\binom{-1-\alpha-\beta}{k}\binom{\beta}{k-j}\bigg)\pmod{p^2}.
\end{align*}
Applying the Saalsch\"utz identity again, we get that
\begin{align*}
&\sum_{j=1}^\alpha\frac{1}{j}\sum_{k=j}^{\alpha}(-1)^{k-j}\binom{\alpha}{k}\binom{-1-\alpha-\beta}{k}\binom{\beta}{k-j}\\
=&\sum_{j=1}^\alpha\frac{(-\alpha)_j(1+\alpha+\beta)_j}{j\cdot(j!)^2}\sum_{k=j}^{\alpha}\frac{(j-\alpha)_{k-j}(-\beta)_{k-j}(j+1+\alpha+\beta)_{k-j}}{((j+1)_{k-j})^2\cdot(k-j)!}\\
=&\sum_{j=1}^\alpha\frac{(-\alpha)_j(1+\alpha+\beta)_j}{j\cdot(j!)^2}\cdot\frac{(j+1+\beta)_{\alpha-j}(-\alpha-\beta)_{\alpha-j}}{(j+1)_{\alpha-j}(-\alpha)_{\alpha-j}}\\
=&\sum_{j=1}^\alpha\frac{(-1)^\beta}{j}\cdot\binom{-\alpha-1}{j+\beta}\binom{\alpha+\beta}{j+\beta}.
\end{align*}
Since $\alpha+\beta\geq p$, we have
\begin{align*}
\sum_{j=\beta+1}^{\alpha+\beta}\frac{(-1)^\beta}{j-\beta}\cdot\binom{-\alpha-1}{j}\binom{\alpha+\beta}{j}
\equiv&
\sum_{j=p}^{\alpha+\beta}\frac{(-1)^\beta}{j-\beta}\cdot\binom{-\alpha-1}{j}\binom{\alpha+\beta}{j}\\
\equiv&
\sum_{j=p}^{\alpha+\beta}\frac{(-1)^\beta}{j-\beta}\cdot\binom{-\alpha-1}{j-p}\binom{\alpha+\beta-p}{j-p}\\
\equiv&
\sum_{j=0}^{\alpha-(p-\beta)}\frac{(-1)^\beta}{j+p-\beta}\cdot\binom{-\alpha-1}{j}\binom{\alpha-(p-\beta)}{j}\pmod{p}.
\end{align*}
\begin{Lem}\label{l1} Suppose that $\alpha,\beta$ are positive integers and $\alpha\geq\beta$. Then
\begin{align}\label{abj}
\sum_{j=0}^{\alpha-\beta}\frac{1}{j+\beta}\cdot\binom{-\alpha-1}{j}\binom{\alpha-\beta}{j}
=\sum_{j=0}^{\beta-1}\frac{(-1)^{1+\alpha}}{j+1+\alpha-\beta}\cdot\binom{-\alpha-1}{j}\binom{\beta-1}{j}.
\end{align}
\end{Lem}
\begin{proof}
We shall use induction on $\alpha$ and $\beta$. When $\beta=1$, we have
\begin{align*}
\sum_{j=0}^{\alpha-1}\frac{1}{j+1}\cdot\binom{-\alpha-1}{j}\binom{\alpha-1}{j}=
-\frac1{\alpha}\sum_{j=0}^{\alpha-1}\cdot\binom{-\alpha}{j+1}\binom{\alpha-1}{\alpha-1-j}=
-\frac1{\alpha}\binom{-1}{\alpha}.
\end{align*}
So (\ref{abj}) is true for $\beta=1$.

Assume that $\beta>1$ and (\ref{abj}) holds for any smaller value of $\beta$.
It is easy to check that
\begin{align*}
&\frac{1}{j+1+\alpha-\beta}\cdot\binom{-\alpha-1}{j}\binom{\beta-1}{j}\\
=&\frac{\alpha+\beta-1}{\alpha^2}\cdot\binom{-\alpha}{j}\binom{\beta-1}{j}+\frac{(\beta-1)^2}{\alpha^2}\cdot\frac1{j+1+\alpha-\beta}\cdot\binom{-\alpha}{j}\binom{\beta-2}{j}.
\end{align*}
It follows for the induction hypothesis that
\begin{align}\label{abb}
&\sum_{j=0}^{\beta-1}\frac{1}{j+1+\alpha-\beta}\cdot\binom{-\alpha-1}{j}\binom{\beta-1}{j}\notag\\
=&\binom{\beta-1-\alpha}{\beta-1}\cdot\frac{\alpha+\beta-1}{\alpha^2}+
\frac{(\beta-1)^2}{\alpha^2}\sum_{j=0}^{\beta-2}\frac{1}{j+1+\alpha-\beta}\cdot\binom{-\alpha}{j}\binom{\beta-2}{j}\\
=&\binom{\beta-1-\alpha}{\beta-1}\cdot\frac{\alpha+\beta-1}{\alpha^2}+
\frac{(\beta-1)^2}{\alpha^2}\sum_{j=0}^{\alpha-\beta}\frac{(-1)^\alpha}{j+\beta-1}\cdot\binom{-\alpha}{j}\binom{\alpha-\beta}{j}\notag
\end{align}
Now letting $n=\alpha-\beta$, we only need to show that
\begin{align}\label{abn}
&\sum_{j=0}^{n}\binom{n}{j}\bigg(\frac{(\alpha-n-1)^2}{j+\alpha-n-1}\cdot\binom{-\alpha}{j}+\frac{\alpha^2}{j+\alpha-n}\cdot\binom{-\alpha-1}{j}\bigg)\notag\\
=&(-1)^{n}(2\alpha-n-1)\cdot\binom{\alpha-1}{n}.
\end{align}
We shall use induction on $n$.
There is nothing to do if $n=0$.
And in view of (\ref{abb}), we have
\begin{align*}
\sum_{j=0}^{n}\frac{1}{j+\alpha-n}\cdot\binom{n}{j}\binom{-\alpha-1}{j}
=\binom{n-\alpha}{n}\cdot\frac{\alpha+n}{\alpha^2}+
\frac{n^2}{\alpha^2}\sum_{j=0}^{n-1}\frac{1}{j+\alpha-n}\cdot\binom{-\alpha}{j}\binom{n-1}{j}.
\end{align*}
So
\begin{align*}
&\sum_{j=0}^{n}\binom{n}{j}\bigg(\frac{(\alpha-n-1)^2}{j+\alpha-n-1}\cdot\binom{-\alpha}{j}+\frac{\alpha^2}{j+\alpha-n}\cdot\binom{-\alpha-1}{j}\bigg)\\
=&\binom{n-\alpha}{n}\cdot(\alpha+n)+\binom{n-\alpha+1}{n}\cdot\frac{(\alpha+n-1)\cdot(\alpha-n-1)^2}{(\alpha-1)^2}\\
&+n^2\sum_{j=0}^{n-1}\frac{1}{j+\alpha-n}\cdot\binom{-\alpha}{j}\binom{n-1}{j}+
\frac{n^2(\alpha-n-1)^2}{(\alpha-1)^2}\sum_{j=0}^{n-1}\frac{1}{j+\alpha-n-1}\cdot\binom{1-\alpha}{j}\binom{n-1}{j}\\
=&\binom{n-\alpha}{n}\cdot(\alpha+n)+\binom{n-\alpha+1}{n}\cdot\frac{(\alpha+n-1)\cdot(\alpha-n-1)^2}{(\alpha-1)^2}\\
&+
\frac{n^2}{(\alpha-1)^2}\cdot (-1)^{n-1}(2\alpha-n-2)\cdot\binom{\alpha-2}{n-1}=(-1)^{n}(2\alpha-n-1)\cdot\binom{\alpha-1}{n}.
\end{align*}
\end{proof}
With help of Lemma \ref{l1}, we get that
\begin{align*}
\sum_{j=1}^\alpha\frac{(-1)^\beta}{j}\cdot\binom{-\alpha-1}{j+\beta}\binom{\alpha+\beta}{j+\beta}
\equiv&\sum_{j=0}^{p-\beta-1}\frac{(-1)^{1+\alpha}}{j+1+\alpha+\beta-p}\cdot\binom{-\alpha-1}{j}\binom{p-\beta-1}{j}\\
\equiv&\sum_{j=1}^\alpha\frac{(-\alpha)_j(-\beta)_j}{j\cdot(j!)^2}\cdot\frac{(j-\alpha-\beta)_{\alpha-j}(1+\beta)_{\alpha-j}}{(j+1)_{\alpha-j}(-\alpha)_{\alpha-j}}\\
=&\sum_{j=1}^\alpha\frac{1}{j}\sum_{k=j}^{\alpha}(-1)^{k-j}\binom{\alpha}{k}\binom{-1-\alpha-\beta}{k}\binom{\beta}{k-j}
\pmod{p}.
\end{align*}
This is the desired result.
\end{proof}

\end{document}